\documentclass{amsart}

\usepackage{hyperref}
\usepackage{amsmath}
\usepackage{amssymb}
\usepackage{graphicx}
\usepackage{mathdots}
\usepackage{gensymb}
\usepackage{epstopdf}

\newtheorem{thm}{Theorem}[section]
\newtheorem{lem}[thm]{Lemma}
\newtheorem{prop}[thm]{Proposition}

\newtheorem{conj}[thm]{Conjecture}

\theoremstyle{definition}

\theoremstyle{remark}
\newtheorem*{remark}{Remark}

\numberwithin{equation}{section}

\DeclareMathOperator{\Cr}{Cr}

\begin{document}


\title{A Curved Brunn-Minkowski Inequality for the Symmetric Group}
\author{Weerachai Neeranartvong}
\address{Department of Mathematics, 
Massachusetts Institute of Technology, USA}
\email{weerachai@mit.edu}

\author{Jonathan Novak}
\address{Department of Mathematics, 
Massachusetts Institute of Technology, USA}
\email{jnovak@math.mit.edu}

\author{Nat Sothanaphan}
\address{Department of Mathematics, 
Massachusetts Institute of Technology, USA}
\email{nsothana@mit.edu}



 \maketitle


\begin{abstract}
In this paper, we construct an injection
$A \times B \rightarrow M \times M$ from the product
of any two nonempty subsets of the symmetric group into
the square of their midpoint set, where the metric is that
corresponding to the conjugacy
class of transpositions.  If $A$ and $B$ are disjoint,
our construction allows to inject two copies of $A \times B$ into
$M \times M$.
These injections imply a positively curved Brunn-Minkowski inequality
for the symmetric group analogous to that obtained by 
Ollivier and Villani for the hypercube.  However,
while Ollivier and Villani's inequality is optimal, we believe
that the curvature term in our inequality can be
improved.  We identify a hypothetical concentration inequality
in the symmetric group and prove that it yields an optimally
curved Brunn-Minkowski inequality.
\end{abstract}

\section{Introduction}
The classical Brunn-Minkowski inequality
may be formulated as follows: given two compact nonempty 
sets $A,B \subset \mathbb{R}^n$, one has

	\begin{equation*}
		\label{ineq:ClassicalBM}
		\log |M_t| \geq (1-t) \log |A| + t \log |B|
	\end{equation*}
	
\noindent
for any $0 \leq t \leq 1$, where 

	\begin{equation*}
		M_t = \{ (1-t)a + tb : a \in A,\ b \in B\}
	\end{equation*} 
	
\noindent
is the set of $t$-midpoints of $A$ and $B$, and $|\cdot|$ is 
Lebesgue measure.  If $\mathbb{R}^n$ is replaced
by a smooth complete Riemannian manifold
with positive Ricci curvature bounded below by $K > 0$, the 
Brunn-Minkowski inequality can be strengthened to

	\begin{equation}
		\label{eqn:CurvedBM}
		\log |M_t| \geq (1-t) \log |A| + t \log |B| + 
		\frac{K}{2}t(1-t)\mathrm{d}(A,B)^2, 
	\end{equation}
	
\noindent
where $\mathrm{d}$ is the Hausdorff distance and
$M_t$ the set of points of the form $\gamma(t)$ with
$\gamma$ a geodesic in $X$ such that 
$\gamma(0) \in A$ and $\gamma(1) \in B$,
see \cite{OV} and references therein.  The heuristic here
is that midpoint sets are larger in positively curved
space than in flat space, and the degree of
this distortion is controlled by Ricci curvature.

Y. Ollivier and C. Villani \cite{OV} have shown
that a curved Brunn-Minkowski inequality analogous
to \eqref{eqn:CurvedBM} holds for the discrete hypercube
$\mathbb{Z}_2^N$ equipped with the Hamming distance.
While the definition of $t$-midpoints
in discrete space is somewhat messy,
in the case $t=\frac{1}{2}$, at least, it is 
reasonable to define the midpoint set $M=M_{1/2}$ of 
$A,B \subseteq \mathbb{Z}_2^N$ to
be the collection of $m \in \mathbb{Z}_2^N$ which satisfy
$\mathrm{d}(a,m) + \mathrm{d}(m,b) = \mathrm{d}(a,b)$ 
and $\mathrm{d}(a,m) = 
\mathrm{d}(m,b) + \varepsilon$, with $\varepsilon \in \{-1,0,1\}$,
for some $(a,b) \in A \times B$.
Adopting these definitions, Ollivier and Villani proved 
the following curved Brunn-Minkowski inequality
for the hypercube \cite[Theorem 1]{OV}.

	\begin{thm}
	\label{thm:OV}
		For any nonempty sets $A,B \subseteq \mathbb{Z}_2^N$,
		
			\begin{equation*}
				\log |M| \geq \frac{1}{2} \log |A| + 
				\frac{1}{2} \log |B| + \frac{K}{8} 
				\mathrm{d}(A,B)^2,
			\end{equation*}
			
		\noindent
		where $K=\frac{1}{2N}$.
	\end{thm}
	
\noindent
Ollivier and Villani moreover verify that 
the dependence of $K$ on $N$ in their result is optimal.
As discussed in \cite{OV}, this result supports the 
statement that the ``discrete Ricci curvature'' of 
$\mathbb{Z}_2^N$ is of order $N^{-1}$.  

For any $n \geq 2N$, there is an injective 
group homomorphism

    \begin{equation*}
        \mathbb{Z}_2^N \longrightarrow S(n)
    \end{equation*}
    
\noindent 
from the $N$-dimensional hypercube into the
symmetric group of rank $n$ determined by

    \begin{equation*}
    \begin{split}
         e_1 &\mapsto (1 \mapsto 2) \\ 
         e_2 &\mapsto (3 \mapsto 4) \\
         &\vdots \\
         e_N &\mapsto (2N-1 \mapsto 2N),
        \end{split}
    \end{equation*}
    
\noindent
where $e_i \in \mathbb{Z}_2^N$ is the bitstring in which
all bits are zero except the $i$th bit, and 
$(2i-1 \mapsto 2i) \in S(n)$ is the transposition
which swaps $2i$ and $2i-1$.  
If one equips $S(n)$ with the metric induced
by the word norm corresponding to the conjugacy class
of transpositions, this injection is an isometric embedding
of $\mathbb{Z}_2^N$ in $S(n)$.  It is thus natural
to seek an extension of Theorem \ref{thm:OV} to the symmetric
group, viewed as a metric space in this way.  
In this paper, we prove the following 
curved Brunn-Minkowski inequality
for $S(n)$.

	\begin{thm}
		\label{thm:Main}
		For any nonempty sets $A,B \subseteq S(n)$,
		
			\begin{equation*}
				\log |M| \geq \frac{1}{2} 
				\log |A| + \frac{1}{2} \log |B| + \frac{K}{8} 
				\mathrm{d}(A,B)^2,
			\end{equation*}
			
		\noindent
		where $K=\frac{4 \log 2}{(n-1)^2}$.
	\end{thm}
	
The Brunn-Minkowski inequality presented in Theorem \ref{thm:Main}
is only slightly curved, and we believe that 
Theorem \ref{thm:Main} in fact holds
with $K=\frac{c}{n-1}$, $c$ a positive constant.  Although
we do not prove this,
we identify a hypothetical concentration inequality
in the symmetric group which generalizes the hypercube
concentration inequality of Ollivier and Villani
\cite[Corollary 6]{OV}, and demonstrate that  
it implies an optimally curved Brunn-Minkowski 
inequality for the symmetric group.

\subsubsection*{Acknowledgement}
We thank Prasad Tetali for introducing
us to the subject of discrete curvature, and suggesting
the problem of lifting the curved Brunn-Minkowski 
inequality from the hypercube to the symmetric group.
We are grateful to Professor Tetali
for communicating to us a set of notes \cite{Tetali} on this topic
compiled by a group of researchers following
an AIM workshop on discrete curvature,
and for pointing out the papers \cite{EMT,KKRT}.  We 
thank Mike Lacroix for creating Figure \ref{fig:Cayley}.

\section{Symmetric Group Basics}
In this section we fix basic notation and terminology
concerning the symmetric group $S(n)$.  We
identify $S(n)$ with its right Cayley graph as generated
by the conjugacy class of transpositions.
Thus $a,b,c,\dots \in S(n)$ are the 
vertices of our graph, and $\{a,b\}$ is an edge
if and only if $a^{-1}b$ fixes all but two
points of $\{1,\dots,n\}$.  In this way $S(n)$
becomes a \emph{graded graph}: it decomposes
as the disjoint union

    \begin{equation*}
        S(n) = \bigsqcup_{r=0}^{n-1} L_r,
    \end{equation*}
    
\noindent
where $L_r$ is the set of permutations 
which factor into exactly $n-r$ disjoint cycles; each
$L_r$ is an independent set; and, finally, 
there exists an edge
between $L_r$ and $L_{r'}$ if and only if
$|r-r'|=1$.  Figure \ref{fig:Cayley} shows the
case $n=4$.  

\begin{figure}
	\includegraphics{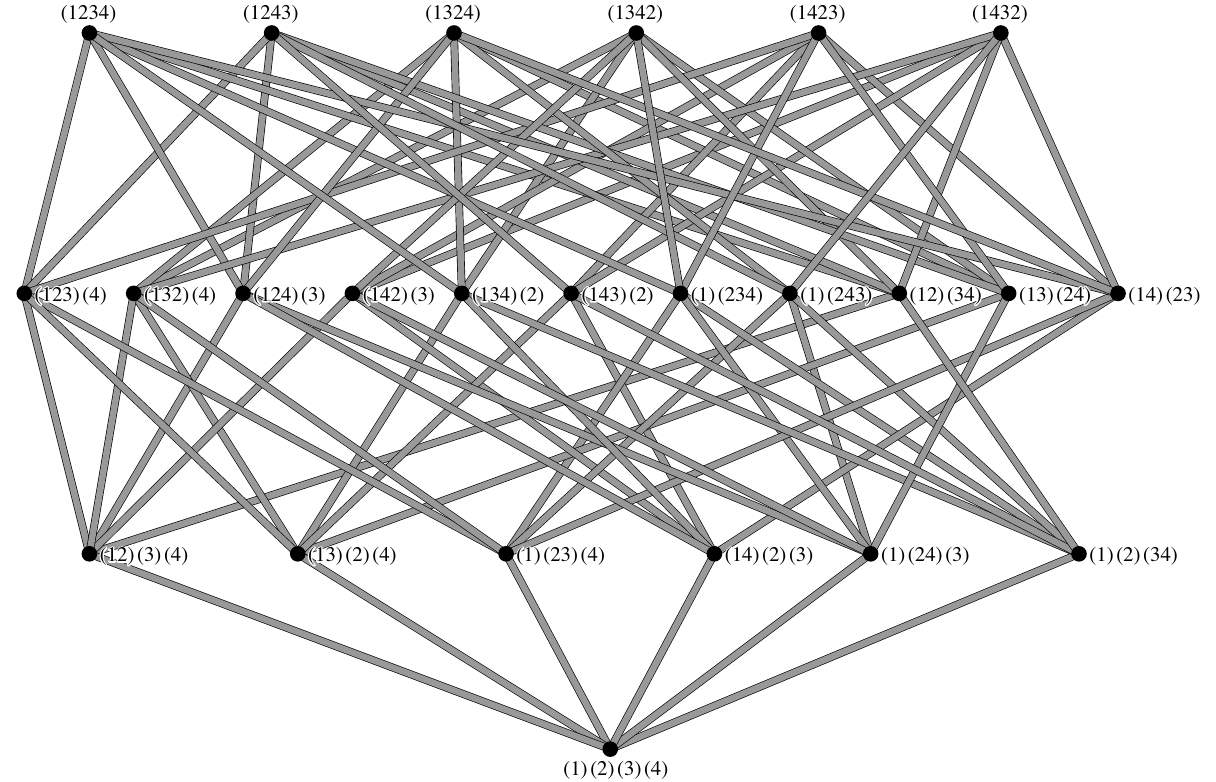}
	\caption{\label{fig:Cayley} The Cayley graph of $S(4)$.}
\end{figure}

Each level $L_r$ of $S(n)$ further decomposes 
as the disjoint union

    \begin{equation*}
        L_r = \bigsqcup_{\substack{\lambda \vdash n\\ 
        \ell(\lambda)=n-r} } C_\lambda,
    \end{equation*}
    
\noindent
where the union is over partitions $\lambda$ of $n$
with $n-r$ parts, and $C_\lambda$ is the set of permutations
with cycle type $\lambda$.  The sets $C_\lambda$ are
the conjugacy classes of $S(n)$.

In this paper we make use of a decomposition
of $S(n)$ which is finer than the usual decomposition
into conjugacy classes.  Given $p \in S(n)$, 
factor $p$ into disjoint cycles, 
and present each cycle so that its leftmost element
is its minimal element.  That is, each cycle of $p$
is presented in the form 

    \begin{equation*}
        (i_1 \mapsto i_2 \mapsto i_3 \mapsto \dots),
    \end{equation*}
    
\noindent
where $i_1 < \min\{i_2,i_3,\dots\}$.
Next, list the cycles of $p$ from left to right
in increasing order of their minimal elements.  Thus
$p$ is presented in the form

    \begin{equation*}
        p = (i_1 \mapsto i_2 \mapsto i_3 \mapsto \dots)
        (j_1 \mapsto j_2 \mapsto j_3 \mapsto \dots)
        (k_1 \mapsto k_2 \mapsto k_3 \mapsto \dots)
        \dots,
    \end{equation*}
    
\noindent
where $i_1 < j_1 < k_1 < \dots$.  We call this
the \emph{ordered cycle factorization} of $p$.
Figure \ref{fig:Cayley} displays the elements of $S(4)$ 
using their ordered cycle factorizations.
We refer to the vector

    \begin{equation*}
        (i_1,j_1,k_1,\dots)
    \end{equation*}

\noindent
as the \emph{sequence of cycle minima} of $p$.
The vector

    \begin{equation*}
      (\mu_1,\mu_2,\mu_3,\dots)
    \end{equation*}

\noindent
of cycle lengths

    \begin{equation*}
        p = 
        (\underbrace{i_1 \mapsto i_2 \mapsto i_3 
        \mapsto \dots}_{\text{length }\mu_1})
        (\underbrace{j_1 \mapsto j_2 \mapsto j_3 
        \mapsto \dots}_{\text{length }\mu_2})
        (\underbrace{k_1 \mapsto k_2 \mapsto k_3 
        \mapsto \dots}_{\text{length }\mu_3})
        \dots,
    \end{equation*}
    
\noindent
in the ordered cycle factorization of $p$
is a composition of $n$ which we call the
\emph{ordered cycle type} of $p$.  We denote
by $\ell(\mu)$ the number of parts of $\mu$,
so that $r=n-\ell(\mu)$ if $p \in L_r$.
Given two permutations $p,p'$ of the same ordered
cycle type, there is a \emph{unique} permutation 
$u$ which both conjugates $p$ into $p'$ and
transforms the sequence of cycle minima of $p$
into that of $p'$.

Given a composition
$\mu \vDash n$, we denote by $\vec{C}_\mu$
the set of all permutations
whose ordered cycle type is $\mu$.  Then each conjugacy
class $C_\lambda$ in $S(n)$ decomposes as the disjoint union

    \begin{equation*}
        C_\lambda = \bigsqcup_\mu \vec{C}_\mu,
    \end{equation*}
    
\noindent 
where $\mu$ ranges over all compositions obtained
by permuting the parts of $\lambda$.  For the 
symmetric group $S(4)$, the successive decompositions
we have discussed are

    \begin{equation*}
    \begin{split}
        S(4) &= L_0 \sqcup L_1 \sqcup L_2 \sqcup L_3 \\
        & = C_{(1,1,1,1)} \sqcup C_{(2,1,1)}
        \sqcup C_{(3,1)} \sqcup C_{(2,2)} \sqcup C_{(4)} \\
        & = \vec{C}_{(1,1,1,1)} \sqcup 
        \vec{C}_{(2,1,1)} \sqcup \vec{C}_{(1,2,1)} \sqcup
        \vec{C}_{(1,1,2)}
        \sqcup \vec{C}_{(3,1)} \sqcup \vec{C}_{(1,3)}
         \sqcup \vec{C}_{(2,2)} \sqcup \vec{C}_{(4)}.
    \end{split}
    \end{equation*}

\noindent
Each class $\vec{C}_\mu$ contains a canonical
permutation $p_\mu$, which acts by cyclically
permuting the first $\mu_1$ positive integers
in the canonical way, cyclically permuting the next
$\mu_2$ positive integers in the canonical way, 
and so on.  Given $p \in \vec{C}_\mu$, we denote by
$u_p \in S(n)$ the unique permutation which
both conjugates $p_\mu$ to $p$ and transforms
the sequence of cycle minima of $p_\mu$ into 
that of $p$.

We equip $S(n)$ with the graph theory 
distance $\mathrm{d}$.
Thus level $L_r$ in the Cayley graph coincides with the 
sphere of radius $r$ centred at the identity permutation
$e \in S(n)$.  The following properties of $\mathrm{d}$
are easily checked:

	\begin{equation*}
		\begin{split}
		\mathrm{d}(a,b) &=\mathrm{d}(pap^{-1},pbp^{-1})\\ 
		&=\mathrm{d}(ab^{-1},e) \\
		&=\mathrm{d}(e,a^{-1}b) \\
		&=\mathrm{d}(ap,bp)\\
		&= \mathrm{d}(pa,pb).
		\end{split}	
	\end{equation*}
	
\noindent
In particular, the diameter of the Cayley graph
is

    \begin{equation*}
        \label{eqn:Diameter}
        \max \{\mathrm{d}(a,b) : a,b \in S(n)\}
        = \max \{\mathrm{d}(e,p) : p \in S(n)\}
        =n-1.
    \end{equation*}
    
We have already mentioned the fact that the set of
permutations which lie on a geodesic path from 
the identity permutation $e$ to an
involution $v$ is isometrically isomorphic to 
a hypercube whose dimension
is half the size of the support of $v$.  We will
also make use of the fact that a permutation lies
on a geodesic path from $e$ to a forward cycle $f$
if and only if it is a product of forward cycles
which together induce a noncrossing partition 
of the support of $f$.
A proof of this folklore result may be found in
\cite[Lecture 23]{NS}.

\section{Midpoint Calculus}
In this section, we generalize the encoding/decoding
formalism of Ollivier and Villani from the hypercube to the 
symmetric group.  Where possible, we try to be consistent
with the notation and terminology of \cite{OV}.

\subsection{Crossovers}
Let $(a,b) \in S(n) \times S(n)$ be a pair of permutations,
and let $M(a,b)$ be the corresponding midpoint set.  
Our first observation is that $M(a,b)$ is
the isometric image of a ``standard'' set of midpoints.
More precisely, let $\mu$ be the ordered cycle type of 
$a^{-1}b$, and let $\Cr(\mu)$ denote the set of midpoints
of $e$, the identity permutation, and $p_\mu$,
the canonical permutation of ordered cycle type
$\mu$.  Adopting the terminology of
\cite{OV}, we call the elements of $\Cr(\mu)$
\emph{crossovers}, or $\mu$-\emph{crossovers}
to be precise.  Consider the function

    \begin{equation*}
        S(n) \longrightarrow S(n)
    \end{equation*}
    
\noindent
defined by

	\begin{equation*}
		 x \mapsto a u_{a^{-1}b} x u_{a^{-1}b}^{-1},
	\end{equation*}
	
\noindent
where $u_{a^{-1}b}$ is the unique permutation which
conjugates $p_\mu$ to $a^{-1}b$ and transforms the 
sequence of cycle minima of $p_\mu$ into that of $a^{-1}b$.
This function is an isometry,
being composed of rotation by $u_{a^{-1}b}$ followed
by translation by $a$.  Moreover, under this
mapping

	\begin{equation*}
		e \mapsto a, \quad p_\mu \mapsto b.
	\end{equation*}
	
\noindent
Thus the mapping restricts to a bijection

    \begin{equation*}
     \Cr(\mu) \longrightarrow M(a,b).
    \end{equation*}
    
\noindent
We write 

    \begin{equation*}
     \varphi_c(a,b) = au_{a^{-1}b}cu_{a^{-1}b}^{-1},
    \end{equation*}

\noindent
and view

    \begin{equation*}
        \varphi_c(a,b), \quad c \in \Cr(\mu)
    \end{equation*}
    
\noindent
as a parameterization of the locus $M(a,b)$
by a ``standard'' set of midpoints.
Following the terminology of \cite{OV}, we
call $\varphi_c(a,b) \in M(a,b)$ the midpoint of
$a$ and $b$ \emph{encoded} by the crossover
$c \in \Cr(\mu)$.

\subsection{Duality}
There is a natural geometric operation 
on crossovers: we view a crossover as the lower
half of a geodesic path from the identity to the
canonical permutation with a given ordered cycle type,
and map it to the corresponding upper half.  
More precisely, given $c \in \Cr(\mu)$, its 
\emph{dual} $c^\vee$ is defined by

	\begin{equation*}
		c^\vee := c^{-1}p_\mu.
	\end{equation*}

We now establish some technical properties of
the operation $c \mapsto c^\vee$ which will be needed below.
First is the basic but important closure property.

    \begin{prop}
    \label{prop:Closed}
        $\Cr(\mu)$ is closed under taking duals.
    \end{prop}
    
    \begin{proof}
    Let $c$ be a midpoint of $e$ and $p_\mu$;
    we have to check that $c^\vee$ is a midpoint
    of $e$ and $p_\mu$.  We have
    
    \begin{equation*}
        \mathrm{d}(e,c^\vee) = \mathrm{d}(e,c^{-1}p_\mu)
        =\mathrm{d}(c,p_\mu)
    \end{equation*}
    
    \noindent
    and
    
    \begin{equation*}
        \mathrm{d}(c^\vee,p_\mu) = \mathrm{d}(c^{-1}p_\mu,p_\mu)
        =\mathrm{d}(c^{-1},e)
        =\mathrm{d}(e,c).
    \end{equation*}
    
    \noindent
    Thus
    
    \begin{equation*}
        \mathrm{d}(e,c^\vee)+\mathrm{d}(c^\vee,p_\mu)
        =\mathrm{d}(e,c) + \mathrm{d}(c,p_\mu)
        =\mathrm{d}(e,p_\mu),
    \end{equation*}
    
    \noindent
    and
    
    \begin{equation*}
        \mathrm{d}(e,c^\vee) - \mathrm{d}(c^\vee,p_\mu)
        = \mathrm{d}(c,p_\mu) - \mathrm{d}(e,c)
        =\varepsilon
    \end{equation*}
    
    \noindent
    with $\varepsilon \in \{-1,0,1\}$, as required.
    \end{proof}
    
    Note that while the map $\Cr(\mu) \rightarrow \Cr(\mu)$
    defined by $c \mapsto c^\vee$ is bijective, it is 
    not involutive: the dual of the dual of $c$ is
    $c$ conjugated by $p_\mu^{-1}$.  For future use, we
    extend the duality operation from points to sets:
    given $C \subseteq \Cr(\mu)$, we define
    
    \begin{equation*}
        C^\vee := \{ c^\vee : c \in C\}.
    \end{equation*}
 	
 	Next is the following important property of crossover
 	duals.
 	
 	\begin{lem} \label{lem:LemForDuality}
    If $c \in \Cr(\mu)$, then $c^{-1}c^\vee$
    has both the same ordered cycle type and
    the same sequence of cycle minima as $p_\mu$.
    \end{lem}

\begin{proof}
First note that $c^{-1}c^\vee=c^{-2}p_\mu$.  
Since $c$ lies on a geodesic path linking 
$e$ to $p_\mu$, each cycle of $c$ is a 
subcycle of some cycle of $p_\mu$. 
Since the cycles of $p_\mu$ are intervals,
our task reduces to proving the following general
statement: whenever $c_1c_2\dots$ is a product 
of forward cycles which induce a noncrossing partition of
$\{1,\dots,k\}$, the product

    \begin{equation*}
    c_1^{-2}c_2^{-2}\dots
    (1\mapsto \dots \mapsto k)
    \end{equation*}
    
\noindent
is a cyclic permutation of the numbers $1,\dots,k$.
If this statement holds, then left multiplication of 
$p_\mu$ by $c^{-2}$ will change neither the 
cycle structure nor cycle minima of $p_\mu$.

Note that we can assume each cycle $c_i$ is of
length at least two.
Suppose that the first cycle, $c_1$, is

    \begin{equation*}
        c_1 = (i_1 \mapsto i_2 \mapsto \dots \mapsto i_{k_1}),
    \end{equation*}
    
\noindent
where $1 \leq i_1 < i_2 < \dots < i_{k_1} \leq k$.
Let us write the full forward $k$-cycle in the form

    \begin{equation*}
        (1\mapsto \dots \mapsto k)=(i_1 \mapsto I_1
        \mapsto i_2 \mapsto I_2 \mapsto \dots 
        \mapsto i_{k_1} \mapsto I_{k_1}),
    \end{equation*}

\noindent
where the $I_*$'s are intervals.  Since

    \begin{align*}
        (i_1 \mapsto I_1
        \mapsto i_2 \mapsto I_2 \mapsto \dots 
        \mapsto i_{k_1} \mapsto I_{k_1})
                 &=(i_1\mapsto \dots \mapsto i_{k_1})
         (i_1 \mapsto I_1) (i_2 \mapsto I_2) \dots (i_{k_1} \mapsto I_{k_1}),
    \end{align*}

\noindent
we have

    \begin{align*}
        (i_1\mapsto \dots \mapsto i_{k_1})^{-2}
        (1\mapsto \dots \mapsto k)
        &=(i_1\mapsto \dots \mapsto i_{k_1})^{-1}
        (i_1 \mapsto I_1) (i_2 \mapsto I_2) \dots (i_{k_1} \mapsto I_{k_1})\\
        &=(i_{k_1}\mapsto \dots \mapsto i_1)
       (i_{k_1} \mapsto I_{k_1}) \dots (i_2 \mapsto I_2) (i_1 \mapsto I_1) \\
        &=(i_{k_1} \mapsto I_{k_1} \mapsto \dots
         \mapsto i_2 \mapsto I_2 \mapsto i_1 \mapsto I_1).
    \end{align*}
    
\noindent
Now, since the cycles $c_1,c_2,\dots$ induce a noncrossing
partition of $\{1,\dots,k\}$, the cycle $c_2$ is contained
in one of the intervals $I_1,\dots,I_{k_1+1}$.  Thus the 
same argument applies to compute

    \begin{equation*}
     c_2^{-2}(i_{k_1} \mapsto I_{k_1} \mapsto \dots
         \mapsto i_2 \mapsto I_2 \mapsto i_1 \mapsto I_1).
    \end{equation*}

\end{proof}

\subsection{Encoding}
The duality operation has been introduced, and its basic
properties developed, in order to make available a 
structured means of encoding \emph{pairs} of midpoints 
using crossovers.
To this end, we introduce the mapping
 	
 	\begin{equation*}
     	\Cr(\mu) \longrightarrow M (a,b) \times M(a,b)
 	\end{equation*}
 	
 	\noindent
 	defined by

 	\begin{equation*}
     	c \mapsto \Phi_c(a,b) :=
     	(\varphi_c(a,b), \varphi_{c^\vee}(a,b)).
 	\end{equation*}
 	
 	\noindent
 	Thus $(x,y)=\Phi_c(a,b)$ is the pair
 	of midpoints of $a$ and $b$ encoded 
 	by $c$ and $c^\vee$.
 	The duality relationship
 	between $c$ and $c^\vee$ induces algebraic and 
 	geometric relations between the pairs
 	$(a,b)$ and $(x,y)$ which may be collectively called
 	\emph{duality of midpoints}.
 	
 	\begin{prop}
 	\label{prop:Duality}
     	If $(x,y) = \Phi_c(a,b)$, then:
     	
     	\begin{enumerate}
     	
     	\smallskip
     	\item
     	$x^{-1}y$ has the same ordered cycle 
     	type and sequence of cycle minima as $a^{-1}b$;
     	
     	\smallskip
     	\item
     	$a$ and $b$ are midpoints of $x$ and $y$;
     	
     	\smallskip
     	\item
     	$u_{x^{-1}y} = u_{a^{-1}b} u_{c^{-1}c^\vee}$.
     	\end{enumerate}
 	\end{prop}

 	\begin{proof}
 	Let us prove these assertions in order.
 	
 		\begin{enumerate}
 		
     		\smallskip
 		    \item 
 			First, we have
  			 \begin{align*}
				x^{-1}y &= (\varphi_c(a,b))^{-1}(\varphi_{c^\vee}(a,b)) \\
				&=(au_{a^{-1}b}cu_{a^{-1}b}^{-1})^{-1}
   				(au_{a^{-1}b}c^\vee u_{a^{-1}b}^{-1}) \\
				&= u_{a^{-1}b}c^{-1}c^\vee u_{a^{-1}b}^{-1},
   			 \end{align*}

\noindent
Let $\mu$ be the ordered cycle type of $a^{-1}b$.
By definition, $u_{a^{-1}b}$
conjugates $p_\mu$ into $a^{-1}b$ and transforms
the sequence of cycle minima of $p_\mu$ into that of $a^{-1}b$.
By Lemma \ref{lem:LemForDuality}, $c^{-1}c^\vee$
has both the same ordered cycle type and
sequence of cycle minima as $p_\mu$.  Thus, 
$u_{a^{-1}b}c^{-1}c^\vee u_{a^{-1}b}^{-1}$ has both the same
ordered cycle type and sequence of cycle minima as $a^{-1}b$.  

\smallskip
\item
We now show that $a$ and $b$ are midpoints of $x$ and $y$.
Since 

    \begin{equation*}
        a^{-1}x=a^{-1}(au_{a^{-1}b}cu_{a^{-1}b}^{-1})
        =u_{a^{-1}b}cu_{a^{-1}b}^{-1}
    \end{equation*}

\noindent
 and 
 
    \begin{align*}
    	y^{-1}b&=(au_{a^{-1}b}c^\vee u_{a^{-1}b}^{-1})^{-1}b \\
    	&=(au_{a^{-1}b}c^{-1}p_\mu u_{a^{-1}b}^{-1})^{-1}b \\
    	&=(au_{a^{-1}b}c^{-1}u_{a^{-1}b}^{-1}
        	u_{a^{-1}b}p_\mu u_{a^{-1}b}^{-1})^{-1}b \\
    &=(au_{a^{-1}b}c^{-1}u_{a^{-1}b}^{-1}a^{-1}b)^{-1}b \\
    &= (b^{-1}au_{a^{-1}b})c(b^{-1}au_{a^{-1}b})^{-1},
    \end{align*}

\noindent
both $a^{-1}x$ and $y^{-1}b$ are conjugates of $c$.
By the conjugation invariance of $\mathrm{d}$ we thus
have

    \begin{equation*}
         \mathrm{d}(e,a^{-1}x) = \mathrm{d}(e,c) =
         \mathrm{d}(e,y^{-1}b),
    \end{equation*}
    
\noindent
whence

    \begin{equation*}
         \mathrm{d}(a,x) = \mathrm{d}(e,c) =
         \mathrm{d}(y,b).
    \end{equation*}

\noindent
Similarly, since

    \begin{equation*}
        a^{-1}y=a^{-1}(au_{a^{-1}b}c^\vee u_{a^{-1}b}^{-1})
        =u_{a^{-1}b}c^\vee u_{a^{-1}b}^{-1}
    \end{equation*}

\noindent
and 
\begin{align*}
	x^{-1}b&=(au_{a^{-1}b}c u_{a^{-1}b}^{-1})^{-1}b \\
	&=u_{a^{-1}b}c^{-1}u_{a^{-1}b}^{-1}a^{-1}b\\ 
	&=u_{a^{-1}b}c^\vee p_\mu^{-1}u_{a^{-1}b}^{-1}a^{-1}b \\ 
	&=u_{a^{-1}b}c^\vee 
    	(u_{a^{-1}b}^{-1}(a^{-1}b)^{-1}u_{a^{-1}b})u_{a^{-1}b}^{-1}a^{-1}b \\ 
	&=u_{a^{-1}b}c^\vee u_{a^{-1}b}^{-1},
\end{align*}

\noindent
both $a^{-1}y$ and $x^{-1}b$ are conjugates of $c^\vee$,
and we have 

    \begin{equation*}
     \mathrm{d}(a,y)=\mathrm{d}(e,c^\vee)=\mathrm{d}(x,b).
     \end{equation*}

\noindent
Now, since $x^{-1}y$ has the same ordered cycle type as
$a^{-1}b$, we have

    \begin{equation*}
        \mathrm{d}(x,y) = \mathrm{d}(a,b) = \mathrm{d}(e,p_\mu).
    \end{equation*}

\noindent
Putting this all together, we have

    \begin{align*}
    \mathrm{d}(x,a) + \mathrm{d}(a,y) 
    &= \mathrm{d}(e,c) + \mathrm{d}(e,c^\vee) \\
    &= \mathrm{d}(e,c) + \mathrm{d}(c,p_\mu) \\
    &= \mathrm{d}(e,p_\mu) \\
    &= \mathrm{d}(x,y),
    \end{align*}

\noindent
and

    \begin{align*}
    \mathrm{d}(x,a) - \mathrm{d}(a,y) 
    &= \mathrm{d}(e,c) - \mathrm{d}(e,c^\vee) \\
    &= \mathrm{d}(e,c) - \mathrm{d}(c,p_\mu) \\
    &= \varepsilon,
    \end{align*}

\noindent
where $\varepsilon \in \{-1,0,1\}$ because
$c$ is a midpoint of $e$ and $p_\mu$.  Thus, $a$
is a midpoint of $x$ and $y$.  The proof that
$b$ is a midpoint of $x$ and $y$ is just the same.

\smallskip
\item
Finally, we prove the identity 

    \begin{equation*}
        u_{x^{-1}y} = u_{a^{-1}b} u_{c^{-1}c^\vee},
    \end{equation*}
    
\noindent
which will be needed below in the proof of Proposition
\ref{prop:Decoding}.  Since

	\begin{equation*}
		\begin{split}
		x^{-1}y &= (au_{a^{-1}b} c u_{a^{-1}b}^{-1})^{-1} 
		(au_{a^{-1}b} c^\vee u_{a^{-1}b}^{-1}) \\
		&= u_{a^{-1}b} c^{-1}c^\vee u_{a^{-1}b}^{-1} \\
		&= u_{a^{-1}b} u_{c^{-1}c^\vee} p_\mu u_{c^{-1}c^\vee} ^{-1}
		u_{a^{-1}b}^{-1} \\
		&= (u_{a^{-1}b} u_{c^{-1}c^\vee}) p_\mu (u_{a^{-1}b} 
		u_{c^{-1}c^\vee})^{-1},
		\end{split}
	\end{equation*}
	
\noindent
we have that $u_{a^{-1}b} u_{c^{-1}c^\vee}$ conjugates
$p_\mu$ into $x^{-1}y$.  This is one of the two properties
that uniquely defines the permutation $u_{x^{-1}y}$,
the other being that it transforms the sequence of cycle minima
of $p_\mu$ into the sequence of cycle minima of $x^{-1}y$.  
Since $u_{c^{-1}c^\vee}$ stabilizes the sequence of cycle
minima of $p_\mu$ (by Lemma \ref{lem:LemForDuality}), 
conjugation of $p_\mu$ by 
$u_{a^{-1}b} u_{c^{-1}c^\vee}$ produces a permutation 
whose sequence of cycle minima coincides with that of
$a^{-1}b$, and hence with that of $x^{-1}y$ by Part (1).

\end{enumerate}
\end{proof}
 	
 \subsection{Decoding} 
    Our constructions so far may be thought of
    in cryptographic terms, as follows.
 Alice and Bob wish to transmit messages to
 one another across an insecure channel.  They
 meet at a secure location, and agree on a composition
 $\mu \vDash n$ and a crossover $c \in \Cr(\mu)$
 to be used as a secret encryption key.  Alice
 and Bob then return to their respective locations on opposite
 ends of the channel.  The plaintext
 messages to be transmitted are pairs $(a,b) \in
 S(n) \times S(n)$ such that $a^{-1}b \in \vec{C}_\mu$.
 To send the message $(a,b)$ to Bob, Alice computes the
 ciphertext $(x,y) = \Phi_c(a,b)$ and transmits it to
 Bob across the channel. Bob receives the ciphertext $(x,y)$,
 and wishes to recover the plaintext message.  Our next result 
 proves that there is a well-defined decryption key $\delta(c)$
 such that $(a,b)=\Phi_{\delta(c)}(x,y)$ --- in fact, the
 proof explains how to compute $\delta(c)$.  
 
    \begin{prop}
    \label{prop:Decoding}
        For each composition $\mu \vDash n$,
        there exists a function
        
        \begin{equation*}
            \delta_\mu : \Cr(\mu) \longrightarrow \Cr(\mu)
        \end{equation*}
        
        \noindent
        such that
        
        \begin{equation*}
            \Phi_{\delta_\mu(c)}(\Phi_c(a,b)) = (a,b)
        \end{equation*}

        \noindent
        holds for all $c \in \Cr(\mu)$ and each
        $(a,b) \in S(n) \times S(n)$ verifying
        $a^{-1}b \in \vec{C}_\mu$.  Moreover,
        this function is an involution.
    \end{prop}
    
We call the function $\delta_\mu$ of Proposition
\ref{prop:Decoding} the \emph{decoding function} of
type $\mu$.  In order to lighten the notation, we
will henceforth omit the dependence of $\delta$ on $\mu$.

\begin{proof}
Fix a composition $\mu \vDash n$.
We claim that the corresponding decoding function
is given by

    \begin{equation*}
        \delta(c) := u_{c^{-1}c^\vee}^{-1}c^{-1}u_{c^{-1}c^\vee}.
    \end{equation*}
    
First, let us check that the codomain of $\delta$ is indeed
$\Cr(\mu)$, i.e. that $\delta(c)$ is in fact a midpoint of
$e$ and $p_\mu$.  We have

    \begin{align*}
        \mathrm{d}(e,\delta(c)) &= 
        \mathrm{d}(e,u_{c^{-1}c^\vee}^{-1}c^{-1}u_{c^{-1}c^\vee})\\
        &=\mathrm{d}(e,c^{-1})\\
        &=\mathrm{d}(e,c)
    \end{align*}
    
\noindent
and

    \begin{align*}
        \mathrm{d}(\delta(c),p_\mu)
        &=\mathrm{d}(u_{c^{-1}c^\vee}^{-1}c^{-1}u_{c^{-1}c^\vee},p_\mu)\\
        &=\mathrm{d}(c^{-1},u_{c^{-1}c^\vee}p_\mu u_{c^{-1}c^\vee}^{-1})\\
        &=\mathrm{d}(c^{-1},c^{-1}c^\vee)\\
        &=\mathrm{d}(e,c^\vee)\\
       &=\mathrm{d}(e,c^{-1}p_\mu)\\
        &=\mathrm{d}(c,p_\mu).
    \end{align*}
    
\noindent
Thus, since $c$ is a midpoint of $e$ and $p_\mu$, we have

    \begin{equation*}
        \mathrm{d}(e,\delta(c))+\mathrm{d}(\delta(c),p_\mu)
        =\mathrm{d}(e,c) + \mathrm{d}(c,p_\mu)
        =\mathrm{d}(e,p_\mu)
    \end{equation*}

\noindent
and 

    \begin{equation*}
        \mathrm{d}(e,\delta(c))-\mathrm{d}(\delta(c),p_\mu)
        =\mathrm{d}(e,c) - \mathrm{d}(c,p_\mu)
        =\varepsilon
    \end{equation*}

\noindent
with $\varepsilon \in \{-1,0,1\}$.

Next, let $(a,b) \in S(n) \times S(n)$ be a valid plaintext, 
i.e. a pair of permutations such that $a^{-1}b \in \vec{C}_\mu$,
and let $(x,y) = \Phi_c(a,b)$ be the corresponding ciphertext.  
We then have 

	\begin{equation*}
		\begin{split}
		x &= \varphi_c(a,b) = au_{a^{-1}b} c u_{a^{-1}b}^{-1} \\
		y &= \varphi_c(a,b) = au_{a^{-1}b} c^\vee u_{a^{-1}b}^{-1},
		\end{split}
	\end{equation*}
	
\noindent
By Proposition \ref{prop:Duality}, Part (1), we may 
re-encode $(x,y)$ using $\delta(c)$ as an encryption key,
arriving at a new ciphertext
$(x',y') = \Phi_{\delta(c)}(x,y)$.  We claim that this 
re-encoding decrypts $(x,y)$, i.e. that
$(x',y')=(a,b)$.  Indeed,

    \begin{align*}
        x' &= \varphi_{\delta(c)}(x,y) \\
        &= x u_{x^{-1}y} \delta(c) u_{x^{-1}y}^{-1} \\
        &= (au_{a^{-1}b} c u_{a^{-1}b}^{-1})
        (u_{a^{-1}b}u_{c^{-1}c^\vee})
        (u_{c^{-1}c^\vee}^{-1}c^{-1}u_{c^{-1}c^\vee})
        (u_{a^{-1}b}u_{c^{-1}c^\vee})^{-1} \\
        &= a,
    \end{align*}
    
\noindent
where we made use of Proposition \ref{prop:Duality},
Part (3).  Similarly, we have

    \begin{align*}
        y' &= \varphi_{\delta(c)^\vee}(x,y) \\
        &= x u_{x^{-1}y} \delta(c)^{\vee} u_{x^{-1}y}^{-1} \\
        &= (au_{a^{-1}b} c u_{a^{-1}b}^{-1})
        (u_{a^{-1}b}u_{c^{-1}c^\vee})
        (u_{c^{-1}c^\vee}^{-1}c^{-1}u_{c^{-1}c^\vee})^\vee
        (u_{a^{-1}b}u_{c^{-1}c^\vee})^{-1} \\
        &=(au_{a^{-1}b} c u_{a^{-1}b}^{-1})
        (u_{a^{-1}b}u_{c^{-1}c^\vee})
        (u_{c^{-1}c^\vee}^{-1}cu_{c^{-1}c^\vee}p_\mu)
        (u_{a^{-1}b}u_{c^{-1}c^\vee})^{-1} \\
        &=a u_{a^{-1}b}c^2u_{c^{-1}c^\vee}p_\mu
        u_{c^{-1}c^\vee}^{-1} u_{a^{-1}b}^{-1} \\
        &= a u_{a^{-1}b} cc^\vee u_{a^{-1}b} \\
        &= a u_{a^{-1}b} p_\mu u_{a^{-1}b} \\
        &=b.
    \end{align*}

It remains to show that
\begin{equation*}
\delta: \Cr(\mu) \longrightarrow \Cr(\mu)
\end{equation*}
is an involution.  Let $c$ be a $\mu$-crossover,
let $(a,b)$ be a valid message,
and consider the triple encoding

    \begin{equation*}
        \Phi_{\delta^2(c)}(\Phi_{\delta(c)}(\Phi_c(a,b))).
    \end{equation*}
    
\noindent
Since $\delta(c)$ is the decryption key corresponding
to the encryption key $c$, we have

    \begin{equation*}
        \Phi_{\delta^2(c)}(\Phi_{\delta(c)}(\Phi_c(a,b)))
        = \Phi_{\delta^2(c)}(a,b).
    \end{equation*}

\noindent
On the other hand, since $\delta^2(c)$ is the 
decryption key corresponding to the encryption key 
$\delta(c)$, we also have

    \begin{equation*}
        \Phi_{\delta^2(c)}(\Phi_{\delta(c)}(\Phi_c(a,b)))
        = \Phi_{c}(a,b).
    \end{equation*}
    
\noindent
Thus $\Phi_{\delta^2(c)}(a,b) = \Phi_{c}(a,b)$,
which readily implies $\delta^2(c)=c$.

\end{proof}
	
\section{The Brunn-Minkowski inequality} 

	\subsection{Without a curvature term}
	We now put our encoding-decoding formalism to work.
    We begin by proving the ``flat'' Brunn-Minkowski
    inequality for $S(n)$.
	
    	\begin{thm}
    	\label{thm:BrunnMinkowski}
        	For any nonempty $A,B \subseteq S(n)$, we have
        	
        	\begin{equation*}
            	\log |M| \geq \frac{1}{2}\log|A| + \frac{1}{2}\log|B|,
        	\end{equation*}
        	
        	\noindent
        	where $M$ is the midpoint set of $A$ and $B$.

    	\end{thm}
    	
    	\begin{proof}
    	 	Suppose that there exists an injection
        	
        	\begin{equation*}
            	\Phi:A \times B \longrightarrow
            	M \times M.
        	\end{equation*}
        	
        	\noindent
        	Then
        	
        	\begin{equation*}
            	|M|^2  \geq |A| |B|,
        	\end{equation*}
        	
        	\noindent
        	and taking logs this becomes
        	
        	\begin{equation*}
            	\log |M| \geq 
            	\frac{1}{2}\log|A| + \frac{1}{2}\log|B|.
        	\end{equation*}
    	
        	To construct such an injection,
        	let $\mu_1,\dots,\mu_m$ be an enumeration
        	of the ordered cycle types of the permutations
        	
        	\begin{equation*}
            	a^{-1}b, \quad (a,b) \in A \times B.
        	\end{equation*}
        	
        	\noindent
        	For each $1 \leq i \leq m$, choose 
        	an encryption key $c_i \in \Cr(\mu_i)$, and let
        	$d_i = \delta(c_i)$ be the 
        	corresponding decryption key.  
        	
        	Partition $A \times B$ into classes 
        	
        	\begin{equation*}
            	C_i = \{ (a,b) \in A \times B : a^{-1}b 
            	\in \vec{C}_{\mu_i}\}, \quad 1 \leq i \leq m,
        	\end{equation*}

        	\noindent
        	and consider the map
        	
        	\begin{equation*}
            	\Phi:A\times B \longrightarrow M \times M
        	\end{equation*}
        	
        	\noindent
        	whose restriction to each $C_i$ is given
        	by encryption using key $c_i$:
        	
        	\begin{equation*}
            	\Phi(a,b) = \Phi_{c_i}(a,b).
        	\end{equation*}
        	
        	\noindent
        	We claim that $\Phi$ is injective.
        	Indeed, if 
        	
        	\begin{equation*}
            	\Phi(a,b) = \Phi(a',b')
        	\end{equation*}
        	
        	\noindent
        	for some $(a,b), (a',b') \in A \times B$,
        	then we must have $(a,b), (a',b') \in C_i$
        	for some $1 \leq i \leq m$ 
        	by Proposition \ref{prop:Duality}.
        	Hence
        	
        	\begin{equation*}
            	(a,b) = \Phi_{d_i}(\Phi_{c_i}(a,b)) 
            	= \Phi_{d_i}(\Phi_{c_i}(a',b'))
            	= (a',b'),
        	\end{equation*}
        	
        	\noindent
        	by Proposition \ref{prop:Decoding}.
    	\end{proof}

	\subsection{With a curvature term}
	Theorem \ref{thm:BrunnMinkowski} above was 
	proved by injecting $A \times B$ into
	$M \times M$.  The injection used was 
	a straightforward application of the formalism
	of encoding and decoding developed above.
	We now obtain the curved Brunn-Minkowski inequality
	stated as Theorem \ref{thm:Main} in the introduction
	by injecting two copies of $A \times B$ into
	$M \times M$.  The construction of the required
	injection involves a subtler use of
	encoding/decoding:
	the proof requires choosing two parallel sequences
	of encryption keys which are coupled in a special
	way.
	
	\begin{thm}
    	\label{thm:CurvedBrunnMinkowski}
        	For any nonempty $A,B \subseteq S(n)$, we have
        	
        	\begin{equation*}
            	\log |M| \geq \frac{1}{2}\log|A| + \frac{1}{2}\log|B|
            	+ \frac{K}{8}\mathrm{d}(A,B)^2
        	\end{equation*}
        	
        	\noindent
        	where $M$ is the midpoint set of $A$ and $B$
        	and
        	
        	\begin{equation*}
            	K = \frac{4\log 2}{(n-1)^2}.
        	\end{equation*}
    	\end{thm}
    	
    	\begin{proof}
        	First note that if $\mathrm{d}(A,B)=0$, the 
        	claimed inequality degenerates to the flat
        	Brunn-Minkowski inequality, which we have
        	already proved.  We thus assume that $A$ and
        	$B$ are disjoint.
        	
        	Suppose that there exists an injection
        	
        	\begin{equation*}
        	\Phi:A \times B \times \{0,1\} \longrightarrow
        	M \times M.
        	\end{equation*}
        	
        	\noindent
        	Then
        	
        	\begin{equation*}
            	|M|^2  \geq 2|A| |B|,
        	\end{equation*}
        	
        	\noindent
        	and taking logs this becomes
        	
        	\begin{equation*}
            	\log |M| \geq 
            	\frac{1}{2}\log|A| + \frac{1}{2}\log|B|
            	+ \frac{\log 2}{2}.
        	\end{equation*}

        \noindent
        The desired inequality now follows from
        the fact that the diameter of $S(n)$ is
        $n-1$.
        
        In order to construct $\Phi$ as required,
        let $\mu_1,\dots,\mu_m$
        and $C_1,\dots,C_m$ be as in the proof of
        Theorem \ref{thm:BrunnMinkowski}.  Choose
        a system of encryption keys
        
        \begin{equation*}
            c_i \in \Cr(\mu_i), \quad 1 \leq i \leq m,
        \end{equation*}
        
        \noindent
        and consider a second system of encryption
        keys obtained from the first system by
        setting
        
        \begin{equation*}
            \tilde{c}_i := \delta(\delta(c_i)^\vee),
            \quad 1 \leq i \leq m.
        \end{equation*}
        
        \noindent
        The keys $\tilde{c}_i$ are defined in this way
        so that, by the involutive property of $\delta$, 
        their corresponding decryption keys are
        the duals of the decryption keys of the first system:
        
        \begin{equation}
        \label{eqn:DerivedCrossovers}
            \delta(\tilde{c}_i) = \delta(c_i)^\vee.
        \end{equation}

        We build $\Phi$ from the above data
        as follows.  First, partition
        $A \times B \times \{0,1\}$ into
        the $2m$ sets
        
        \begin{equation*}
            C_i^{(j)} = \{ (a,b,j) : a^{-1}b \in C_i\}, \quad
            1 \leq i \leq m,\ 0 \leq j \leq 1.
        \end{equation*}
        
        \noindent
        We define $\Phi$ by declaring its restriction 
        to $C_i^{(j)}$ to be
        
        \begin{equation*}
            \Phi(a,b,j) := \begin{cases}
            \Phi_{c_i}(a,b), \text{ if } j=0 \\
            \Phi_{\tilde{c}_i}(a,b), \text{ if } j=1
            \end{cases}
        \end{equation*}

        We now prove that $\Phi$ so defined is an
        injection.  Suppose
        
        \begin{equation*}
            (x,y) \in M \times M,\quad
            (a,b,j) \in C_i^{(j)},\quad 
            (a',b',j') \in C_{i'}^{(j')} 
        \end{equation*}
            
        \noindent
        are such that
        
        \begin{equation*}
           (x,y) = \Phi(a,b,j) = \Phi(a',b',j').
        \end{equation*}
        
        \noindent
        Then, by the first
        part of Proposition \ref{prop:Duality}, 
        we must have $i=i'$.  We claim that also
        $j=j'$.  If not, then (relabelling if neccessary)
        we have
        
        \begin{equation*}
            (x,y) = \Phi(a,b,0) = \Phi(a',b',1),
        \end{equation*}
        
        \noindent
        so that
        
        \begin{equation*}
            (x,y) = \Phi_{c_i}(a,b) = \Phi_{\tilde{c}_i}(a',b').
        \end{equation*}
        
        \noindent
        Decoding, we obtain
        
        \begin{equation*}
            (a,b) = \Phi_{\delta(c_i)}(x,y)
            =(\varphi_{\delta(c_i)}(x,y),
            \varphi_{\delta(c_i)^\vee}(x,y))
        \end{equation*}
        
        \noindent
        and, using \eqref{eqn:DerivedCrossovers},
        
        \begin{equation*}
            (a',b') = \Phi_{\delta(\tilde{c}_i)}(x,y)
            =(\varphi_{\delta(c_i)^\vee}(x,y),
            \varphi_{\delta(\tilde{c}_i)^\vee}(x,y)).
        \end{equation*}
        
        \noindent
        This implies $a'=b$, which is impossible
        since $A$ and $B$ are disjoint.

        There are now two cases:
        $j=j'=0$ and $j=j'=1$.  In the first
        case, we have 
        
        \begin{equation*}
             \Phi_{c_i}(a,b)= \Phi(a,b,0) =\Phi(a',b',0)
            =\Phi_{c_i}(a',b'),
        \end{equation*}
        
        \noindent
        whence
        
        \begin{equation*}
            (a,b) = \Phi_{\delta(c_i)}(\Phi_{c_i}(a,b)) 
            =\Phi_{\delta(c_i)}(\Phi_{c_i}(a',b')) =
            (a',b').
        \end{equation*}

     \noindent
     In the second case, we have 
        
        \begin{equation*}
             \Phi_{\tilde{c}_i}(a,b) =\Phi(a,b,1) =\Phi(a',b',1)
            =\Phi_{\tilde{c}_i}(a',b'),
        \end{equation*}
        
        \noindent
        whence
        
        \begin{equation*}
            (a,b) = 
            \Phi_{\delta(\tilde{c}_i)}(\Phi_{\tilde{c}_i}(a,b)) 
            =\Phi_{\delta(\tilde{c}_i)}
            (\Phi_{\tilde{c}_i}(a',b')) =
            (a',b').
        \end{equation*}

    	\end{proof}
	
	\subsection{With an optimal curvature term}
	We conjecture that the curved Brunn-Minkowski inequality
	proved in Theorem \ref{thm:CurvedBrunnMinkowski}
	can be improved to the following optimal statement.
	
	\begin{conj}
	\label{conj:Optimal}
	For any nonempty $A,B \subseteq S(n)$, we have
        	
        	\begin{equation*}
            	\log |M| \geq \frac{1}{2}\log|A| + \frac{1}{2}\log|B|
            	+ \frac{K}{8}\mathrm{d}(A,B)^2
        	\end{equation*}
        	
        	\noindent
        	where $M$ is the midpoint set of $A$ and $B$,
        	
        	\begin{equation*}
            	K = \frac{c}{n-1},
        	\end{equation*}
        	
        \noindent
        and $c$ is a positive constant.
    \end{conj}

    While we are at present unable to prove
    Conjecture \ref{conj:Optimal}, we show 
    here that it is implied by the following
    conjectural concentration inequality.
    
    \begin{conj}
    \label{conj:Concentration}
        There exists a positive constant $\varepsilon > 0$
        such that, for any $\mu \vDash n$, we have
        
        \begin{equation*}
        \mathrm{d}(C,C^\vee) \geq r \implies
        \mathbb{P}(C)
        \leq e^{-\frac{\varepsilon r^2}{n-\ell(\mu)}},
        \end{equation*}
        
        \noindent
        where $\mathbb{P}$ is the uniform probability
        measure on $\Cr(\mu)$.
    \end{conj}
    
    	\begin{remark}
		In the case where $\mu=(\mu_1,\mu_2,\dots)$ 
		satisfies $\mu_i \in \{1,2\}$, Conjecture \ref{conj:Concentration}
		is true --- via the embedding of the 
		hypercube described in the Introduction,
		it is equivalent to Corollary 6 in \cite{OV}.
	\end{remark}

        We now explain how 
        Conjecture \ref{conj:Optimal} can be deduced
        from Conjecture \ref{conj:Concentration}.  
        This argument, which lifts the proof of 
        \cite[Theorem 1]{OV} from the hypercube to the symmetric group, 
        differs substantially
        from the proofs of Theorems \ref{thm:BrunnMinkowski}
        and \ref{thm:CurvedBrunnMinkowski}.  To prove
        these coarser results, we used static encoding
        to construct our injections,
        i.e. a predetermined list of encryption keys.  
        Here, we use an adaptive coding scheme in which
        all crossovers associated to 
        $A \times B$ are employed.

    \begin{thm}
    \label{thm:ConditionalOptimal}
        Suppose that Conjecture \ref{conj:Concentration}
        is true.  Then, for any nonempty $A,B \subseteq S(n)$, we have
        	
        	\begin{equation*}
            	\log |M| \geq \frac{1}{2}\log|A| + \frac{1}{2}\log|B|
            	+ \frac{K}{8}\mathrm{d}(A,B)^2
        	\end{equation*}
        	
        	\noindent
        	where $M$ is the midpoint set of $A$ and $B$
        	and
        	
        	\begin{equation*}
            	K = \frac{4\varepsilon}{n-1}.
        	\end{equation*}
    \end{thm}
    
    \begin{proof}
        Let $\mu_1,\dots,\mu_m$ and $C_1,\dots,C_m$
        be as in the proof of Theorem \ref{thm:BrunnMinkowski}.
        Put
        
            \begin{equation*}
                U_i := C_i \times \Cr(\mu_i), \quad 1\leq 
                i \leq m.
            \end{equation*}
            
        \noindent
        These sets are pairwise disjoint, but note that they 
        are not subsets of $A \times B$.
        Rather,
        
        \begin{equation*}
            U := \bigsqcup_{i=1}^m U_i
        \end{equation*}
        
        \noindent
        is an ``enriched'' version of
        $A \times B$ in which each pair $(a,b) \in C_i$
        appears with multiplicity $|\Cr(\mu_i)|$.
        
        Consider the map
        
        \begin{equation*}
            \Phi: U \longrightarrow M \times M
        \end{equation*}
        
        \noindent
        defined by
        
        \begin{equation*}
            \Phi(a,b,c) := \Phi_c(a,b).
        \end{equation*}
        
        \noindent
        By Proposition \ref{prop:Duality}, we know that
        the image of $U_i$ under $\Phi$ is contained in
        
        \begin{equation*}
            V_i := \{ (x,y) \in M \times M : x^{-1}y \in
            \vec{C}_{\mu_i}\}.
        \end{equation*}
        
        \noindent
        The sets $V_1,\dots,V_m$ are disjoint, and
        
        \begin{equation*}
         \sum_{i=1}^m |V_i| \leq |M \times M|.
        \end{equation*}

        Fix an arbitrary $i \in \{1,\dots,m\}$,
        and an arbitrary pair of midpoints $(x,y) \in V_i$.
        By definition, $(a,b,c) \in U_i$ maps to
        $(x,y)$ under $\Phi$ if and only if
        
        \begin{equation*}
            \Phi_c(a,b) = (x,y),
        \end{equation*}
        
        \noindent
        which by Proposition \ref{prop:Decoding}
        is equivalent to
        
        \begin{equation*}
            \Phi_{\delta(c)}(x,y) = (a,b).
        \end{equation*}
        
        \noindent
        Thus the cardinality of the fibre of $\Phi$
        over $(x,y)$ agrees with that of the set
        
        \begin{equation*}
            D(x,y) := \{d \in \Cr(\mu_i) : \Phi_d(x,y) \in C_i \}.
        \end{equation*}

        \noindent
          Now let $d_1,d_2 \in D(x,y)$.  
          Then, by definition,
          
          \begin{equation*}
              \varphi_{d_1}(x,y) \in A, \quad 
              \varphi_{d_2^\vee}(x,y) \in B.
          \end{equation*}
          
          \noindent
          Because the map
          
          \begin{equation*}
              d \mapsto \varphi_d(x,y)
          \end{equation*}
          
          \noindent
          is an isometry, this implies

          \begin{equation*}
              \mathrm{d}(d_1,d_2^\vee) =
              \mathrm{d}(\varphi_{d_1}(x,y),\varphi_{d_2^\vee}(x,y))
              \geq \mathrm{d}(A,B),
          \end{equation*}
          
          \noindent
          and hence 
          
          \begin{equation*}
              \mathrm{d}(D(x,y),D(x,y)^\vee) 
              \geq \mathrm{d}(A,B).
          \end{equation*}

         \noindent
         Hence, assuming Conjecture \ref{conj:Concentration},
         we have that
         
         \begin{equation*}
             |D(x,y)| \leq 
             e^{-\frac{\varepsilon \mathrm{d}(A,B)^2}{n-\ell(\mu_i)}}
             |\Cr(\mu_i)| \leq 
             e^{-\frac{\varepsilon \mathrm{d}(A,B)^2}{n-1}}
             |\Cr(\mu_i)|.
         \end{equation*}
         
         \noindent
         Summing over all $(x,y) \in V_i$ we thus obtain
         
         \begin{equation*}
             |U_i| \leq 
             e^{-\frac{\varepsilon \mathrm{d}(A,B)^2}{n-1}}
             |\Cr(\mu_i)| |V_i|.
         \end{equation*}
         
         \noindent
         Since $|U_i| = |C_i| |\Cr(\mu_i)|$, this implies
         
         \begin{equation*}
             |C_i| \leq 
             e^{-\frac{\varepsilon \mathrm{d}(A,B)^2}{n-1}}
             |V_i|.
         \end{equation*}
         
         \noindent
         Summing over $1 \leq i \leq m$, we get
         
         \begin{equation*}
             \sum_{i=1}^m|C_i| \leq 
             e^{-\frac{\varepsilon \mathrm{d}(A,B)^2}{n-1}}
             \sum_{i=1}^m|V_i|,
         \end{equation*}
         
         \noindent
         whence
         
         \begin{equation*}
             |A\times B| \leq 
             e^{-\frac{\varepsilon \mathrm{d}(A,B)^2}{n-1}}
             |M \times M|.
         \end{equation*}
         
         \noindent
         Taking logs, we obtain the curved Brunn-Minkowski
         inequality with
         
         \begin{equation*}
            	K = \frac{4\varepsilon}{n-1}.
        	\end{equation*}

    \end{proof}

\bibliographystyle{amsplain}

\begin{thebibliography}{10}

    \bibitem{EMT}
    M. Erbaar, J. Maas, P. Tetali,
    \textit{Discrete Ricci curvature bounds 
    for the Bernoulli-Laplace and random transposition models},
    \textsf{arXiv:1409.8605v1}

	\bibitem{KKRT}
	B. Klartag, G. Kozma, P. Ralli, P. Tetali,
	\textit{Discrete curvature and abelian groups},
	\textsf{arXiv:1501.00516v2}
    
   \bibitem{NS}
    A. Nica, R. Speicher,
    \textit{Lectures on The Combinatorics of Free Probability},
    London Mathematical Society Lecture Note Series, No. 335,
    2006.
    
	\bibitem{OV}
	Y. Ollivier, C. Villani,
	\textit{A curved Brunn-Minkowski inequality on the hypercube, or:
	What is the Ricci curvature of the discrete hypercube?},
	Siam J. Discrete Math. \textbf{26} (2012), 983-996.
	
	
	\bibitem{Tetali}
	K. Costello, N. Gozlan,  J. Melbourne, W. Perkins,
    C. Roberto, P.-M. Samson, P. Tetali,
    \textit{Notes on displacement convexity on the 
    symmetric group}, unpublished.
	
	
\end{thebibliography}

\end{document}